\documentclass[a4paper,12pt]{article}

\usepackage[english]{babel}
\usepackage{amsmath,amsthm}
\usepackage{amsfonts}
\usepackage{amssymb}
\usepackage{latexsym}
\usepackage{ifthen}
\usepackage{german}
\usepackage{natbib}
\usepackage{graphicx}
\usepackage{array}
\usepackage{color}
\usepackage[a4paper,scale=0.75]{geometry}
\usepackage{enumerate}

\theoremstyle{plain}

\newtheorem{lemma}{Lemma}
\newtheorem{theorem}{Theorem}
\newtheorem{cor}{Corollary}

\theoremstyle{definition}

\theoremstyle{remark}
\newtheorem{remark}{Remark}
\newtheorem{exmp}{Example}

\newcommand{\cR}{\mathcal R}
\newcommand{\R}{\mathbb{R}}

\newcommand{\RR}{\mathbb{R}}

\newcommand{\card}[1]{\left|#1\right|}

\title{Recognizing Treelike $k$"=Dissimilarities}

\author{Sven Herrmann \and Katharina T. Huber 
\and Vincent Moulton \and Andreas Spillner}

\begin{document}
\maketitle
\let\thefootnote\relax\footnotetext{K.T.H., S.H. and V.M. thank Charles Semple and Mike Steel for
their hospitality at University of Canterbury, where  part of this
work was undertaken. V.M. also thanks the Royal Society for its support.
S.H. also thanks the German Academic Exchange Service (DAAD) for its support
by a fellowship within its Postdoc-Programme.\\
The authors would like to thank the anonymous referees for their helpful comments.
}

\begin{abstract}
A $k$-dissimilarity $D$ on a finite set $X$, $|X|\ge k$, 
is a map from the set of size $k$ subsets of $X$ to the real numbers. 
Such maps naturally arise from 
edge-weighted trees $T$ with leaf-set $X$: Given a subset $Y$ 
of $X$ of size $k$, $D(Y)$ is defined to be the total length of the 
smallest subtree of $T$ with leaf-set $Y$. In case 
$k=2$, it is well-known that $2$-dissimilarities
arising in this way can be characterized by the 
so-called ``4-point condition''. However, 
in case $k>2$ Pachter and Speyer recently posed
the following question: Given an arbitrary $k$-dissimilarity, how do
we test whether this map comes from a tree?  In this paper, 
we provide an answer to this question, showing that for $k \ge 3$
a $k$-dissimilarity on a set $X$ arises from a tree 
if and only if its restriction to every $2k$-element subset of $X$ 
arises from some tree,
and that $2k$ is the least possible 
subset size to ensure that this is the case.
As a corollary, we show that there exists a polynomial-time 
algorithm to determine when a $k$-dissimilarity arises from a tree.
We also give a $6$-point condition for determining when
a $3$-dissimilarity arises from a tree, that is similar to the 
aforementioned $4$-point condition.
\end{abstract}

\section{Introduction}

In phylogenetics, as well as other areas making 
use of classification techniques,
many distance-based methods for constructing trees are based on
the following fundamental observation. 
For $X$ a non-empty finite set, 
a graph-theoretical tree $T=(V,E)$ with leaf-set $X \subseteq V$ 
and non-negative edge-weighting $\omega:E \to \R$
can be encoded in terms of the restriction of the
{\em pairwise dissimilarity $d_{(T,\omega)}$} to \(X\), where
$d_{(T,\omega)}(u,v)$ denotes the length of the 
shortest path in $T$ between $u$ and $v$ ($u,v \in V$).
In other words, the tree $T$ 
can be completely recovered from the matrix of pairwise 
values $(d_{(T,\omega)})_{x,y \in X}$.
Such dissimilarities are commonly called 
``tree metrics'' and there is an extensive 
literature concerning their properties 
(see e.g.  \cite{sem-ste-03a} and \cite{gor-87} for overviews).

Various methods have been proposed for constructing trees that exploit 
this observation. These essentially work by projecting an arbitrary 
pairwise dissimilarity onto some ``nearby'' tree metric 
(see e.g. \cite{fel-03,desoete-83}). Even so, it is well-known that such methods
can suffer from the fact that pairwise distance estimates 
involve some loss of information (see e.g. page 176 in \cite{fel-03}).
As a potential solution to this problem, \cite{pac-spe-04a} 
proposed using $k$-\emph{wise} distance estimates, $k \ge 3$, to 
reconstruct trees, an approach which they subsequently
implemented in \cite*{lev-yos-06a} (see also
\cite{gri-99a} where a related idea was investigated). Their rationale
was that $k$-wise estimates are potentially more accurate since they 
can capture more information than 
pairwise distances, a point that was also made 
in Chapter~12 of~\cite{fel-03}.

To describe Pachter and Speyer's approach, recall that
a {\em phylogenetic tree (on $X$)} is 
a graph-theoretical tree \(T = (V,E)\)
in which every non-leaf vertex has degree 
at least three and whose leaf-set is \(X\)
(cf. Figure~\ref{figure:rooted:example:intro}(a)).
In case a real-valued weight \(\omega(e)\) is associated 
to every edge \(e\) of \(T\), we call 
$T$ a \emph{weighted phylogenetic tree}, and
we usually denote such a tree  
by \((T,\omega)\).  Now, for
any $k$-element subset $Y \subseteq X$, $k\ge 2$,
let $D^k_{(T,\omega)}(Y)$ denote the total edge-weight of the 
smallest subtree of $T$ with leaf-set $Y$ 
(cf. Figure~\ref{figure:rooted:example:intro}(b)). 
Note that this quantity is 
sometimes called  the ``phylogenetic diversity'' 
of $Y$ (see e.g. \cite{faith:pd:1992} and \cite{steel:greedy}) and that, 
for $k=2$, $D^2_{(T,\omega)}(\{x,y\}) = d_{(T,\omega)}(x,y)$ for
all $x,y \in X$. 

In \cite{pac-spe-04a}, the following result is proven:

\begin{figure}
\centering
\includegraphics[scale=1.0]{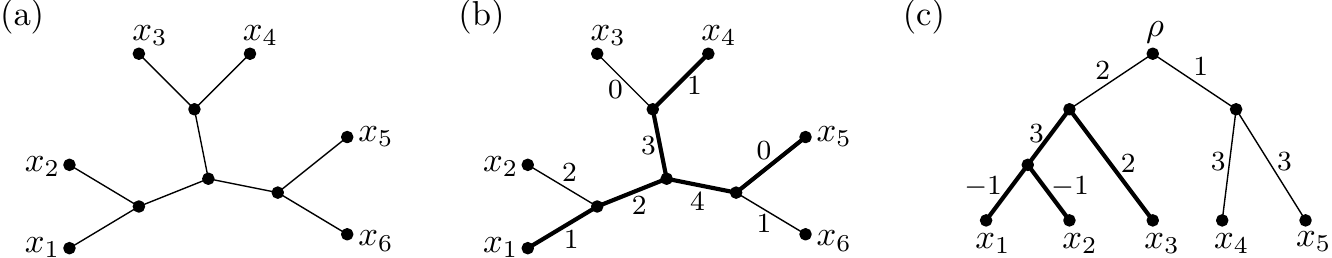}
\caption{(a) A phylogenetic tree \(T=(V,E)\) on
             \(X = \{x_1,x_2,\dots,x_6\}\).
         (b) A non-negative edge-weighting \(\omega\) 
             for the phylogenetic tree \(T\) in (a). The edges
             in the smallest subtree of $T$ containing the set
             \(Y := \{x_1,x_4,x_5\}\) are drawn bold and their
             total weight is \(D^3_{(T,\omega)}(Y) = 11\).
         (c) A weighted, rooted phylogenetic tree \((T,\rho,\omega)\)
             on \(X=\{x_1,x_2,\dots,x_5\}\) with an equidistant
             edge-weighing \(\omega\). The edges whose
             weight contribute to \(D^3_{(T,\omega)}(\{x_1,x_2,x_3\}) = 3\)
             are drawn bold.}
\label{figure:rooted:example:intro}
\end{figure}

\begin{theorem} \label{thm:k-diss-unique}
Let $(T,\omega)$ be a weighted phylogenetic tree on $X$
with \(\omega\) non-negative and \(\omega(e)>0\)
for every edge \(e\) of \(T\) that is not incident
to a leaf, $|X|=n$, and $k \ge 2$ be some integer. 
If $n \ge 2k-1$, then $(T,\omega)$ is determined by the map 
$D^k_{(T,\omega)}$ (and it is not if $2k-2 = n > 2$).
\end{theorem}

In other words, just as in the case $k=2$, for $k \ge 3$
it is possible to recover $(T,\omega)$ from
the function $D^k_{(T,\omega)}$ that maps the set of subsets of $X$
of size $k$ (denoted $\binom{X}{k}$) to~$\RR$.
Here we call any map \(D:\binom{X}{k} \rightarrow \mathbb{R}\)
a {\em $k$"=dissimilarity}. Note that 
$3$"=dissimilarities have been investigated, for example,
in \cite{hay-72a}, \cite{jol-cal-95}
and \cite{hei-ben-97}, and arbitrary $k$"=dissimilarities in
\cite{dez-ros-00} and \cite{war-10}, under names
such as $k$-\emph{way dissimilarities}, 
$k$-\emph{way distances} and $k$"=\emph{semimetrics}
(see also \cite{ban-dre-94} for related work).

In this paper we shall provide a solution to the following
problem raised in \cite{pac-spe-04a}:

\begin{quote}
``However, if we are simply given a $k$"=dissimilarity map
$D :\binom{X}{k} \to \mathbb{R}$, we do not know how to test 
whether this map comes from a phylogenetic tree.''
\end{quote}

Note that \cite{dre-ste-07} study the related problem
of characterizing when a map \(D\) from the set
of subsets of \(X\) of size \emph{at most} $k$ into 
some Abelian group $G$ can be represented by a
phylogenetic tree on $X$ whose edges are assigned
elements from~$G$. However,  we consider 
subsets of \(X\) of size \emph{precisely} \(k\), leading
to a quite different characterization.

In order to state the main result of this paper, we 
first recall some more definitions concerning phylogenetic trees.
A \emph{rooted} phylogenetic tree (on $X$), is a
tree \(T=(V,E)\) with (i) a distinguished 
vertex \(\rho\), called the \emph{root} of \(T\),
that has degree at least~2, (ii) leaf-set \(X\) and
(iii) no vertex in \(V \setminus (X \cup \{\rho\})\)
with degree less than~3.
In case a real-valued weight \(\omega(e)\) is associated 
to every edge \(e \in E \), we call 
$T$ a \emph{weighted}, rooted phylogenetic tree,
and denote it by \((T,\rho,\omega)\). 
Note that, for such a tree, we define 
the maps $d_{(T,\omega)}$ and 
$D^k_{(T,\omega)}$ in the same way as for (unrooted) phylogenetic trees
(cf. Figure~\ref{figure:rooted:example:intro}(c)).
In addition, we call an edge-weighting \(\omega\) of $T$ \emph{equidistant} 
if (i) \(d_{(T,\omega)}(x,\rho) = d_{(T,\omega)}(x',\rho)\)
for all \(x,x' \in X\), and 
(ii) \(d_{(T,\omega)}(x,u) \leq d_{(T,\omega)}(x,v)\)
for all \(x \in X\) and any \(u,v \in V \)
that lie on the path from \(x\) to \(\rho\) in
\(T\) which first meets $u$ and then \(v\) 
(cf. Figure~\ref{figure:rooted:example:intro}(c)).  
Such weightings commonly arise when modeling
sequence evolution assuming a molecular clock
(see e.g. \cite{fel-03}).

Now, we call a $k$-dissimilarity $D$ 
{\em treelike} if there exists a 
weighted phylogenetic tree $(T,\omega)$ 
with \(\omega\) non-negative such that 
$D=D^k_{(T,\omega)}$ holds, and we call $D$ 
\emph{equidistant} if there exists
a weighted, rooted phylogenetic 
tree \((T,\rho,\omega)\) on \(X\) with \(\omega\)
equidistant such that \(D=D^k_{(T,\omega)}\) holds.
In this paper, we shall prove the following:

\begin{theorem}\label{main}
Let $k\geq 2$ and $D$ be a \(k\)"=dissimilarity map 
on a set \(X\) with \(|X| \geq 2k\). Then 
\(D\) is treelike/equidistant if and only if the
restriction of \(D\) to every \(2k\)-element 
subset of \(X\) is treelike/equidistant.
Moreover, for all $k\ge 3$ there exist $k$"=dissimilarity 
maps whose restrictions to every $(2k-1)$-element 
subset of~$X$ are treelike/equidistant but that are not 
treelike/equidistant.
\end{theorem}

Note that for the case \(k=2\) this result is well-known 
(see e.g. \citet[Theorem 7.2.5 and 
Corollary 7.2.7]{sem-ste-03a}).

After presenting some preliminaries in
the next section, we prove Theorem~\ref{main} in
Sections~\ref{sec:treelike} and \ref{sec:examples}. As a 
corollary of Theorem~\ref{main}, we
also show that, for fixed $k \ge 2$,
there is an algorithm with run-time that is 
bounded by a polynomial in \(|X|\)
to decide if an arbitrary $k$-dissimilarity \(D\) 
is treelike (Corollary~\ref{polycheck}). It 
would  be interesting to know if such 
algorithms can be found that have 
good run-time bounds for $k\ge 3$,  
such as those that have been devised for
$k=2$ (see e.g. \cite{cul-rud-89a}, \cite{ban-90}). More generally, 
it might also be of interest to use Theorem~\ref{main}
to help devise new methods to construct trees from
$k$-dissimilarities such as the one 
described in \cite{lev-yos-06a}.

Note that for \(k=2\) the bound $2k = 4$ 
given in the second sentence of Theorem~\ref{main} is 
sharp for treelike dissimilarities, 
but that it can be improved to \(2k-1 = 3\)
for  equidistant dissimilarities 
(see e.g. \citet[Theorem 7.2.5]{sem-ste-03a}).
Although this is not the case for $k\ge 3$, 
in Section~\ref{equisection} we shall prove 
that under certain circumstances it
may still be possible to recover a 
tree from a $k$-dissimilarity $D$ on $X$ in case 
it is equidistant on every \((2k-1)\)-element subset of $X$
(see Theorem~\ref{thm:topology}). 

We conclude the paper by considering 
3-dissimilarities in more detail. 
It is well-known (see e.g. 
\cite{gor-87} and \cite{sem-ste-03a}) that
treelike and equidistant 2-dissimilarities can be
characterized in terms the \emph{4-point}  and
\emph{ultrametric} condition, respectively
(for more details see Section~\ref{sec:k=3}).
Thus, for $k \ge 3$, we can ask for similar ``$m$-point'' 
conditions that characterize 
treelike/equidistant \(k\)"=dissimilarities.
This question has been 
studied in \cite{rub-09a} for the case $k=3$, where
a recursive characterization  
is provided, and
related problems are considered in \cite{boc-col-09a} 
in the context of tropical geometry.
In addition, a necessary (but not sufficient) $(k+2)$-point condition is given for
the general case in \cite[p. 618]{pac-spe-04a}.

In the last section, we provide explicit 
6-point characterizations for 3-dissimilarities
that are treelike/equidistant, which
can be regarded as generalizations of 
the 4"=point/ul\-tra\-metric conditions
(Theorem~\ref{thm:6-point-condition}).
We conclude with a short discussion as to 
why finding similar conditions for $k\ge 4$ appears to 
be somewhat more challenging.\\

\section{Preliminaries on phylogenetic trees}

For the remainder of this paper, $X$ will always
denote a non-empty, finite set.
Also, for a $k$-dissimilarity \(D:\binom{X}{k} \rightarrow \mathbb{R}\) 
and \(\{x_1,x_2,\dots,x_k\} \in \binom{X}{k}\),
we will write \(D(x_1,x_2,\dots,x_k)\)
instead of \(D(\{x_1,x_2,\dots,x_k\})\).

We now recall some further definitions concerning
phylogenetic trees (for more details see \cite{sem-ste-03a}).
Let \(T=(V,E)\) be a phylogenetic tree on \(X\).
A vertex \(v \in V \setminus X\) is called 
an \emph{interior} vertex of~\(T\). The
set of leaves of \(T\), that is, the set \(X\), is also
denoted by \(L(T)\). 
Recall that it is assumed that all interior vertices have degree at
least three. An edge \(e \in E\) is
called \emph{pendant} if it is incident to a leaf of~\(T\).
All other edges are called \emph{interior} edges.

Now, two phylogenetic trees \(T_1 = (V_1,E_1)\) and \(T_2=(V_2,E_2)\)
on the same set \(X\) are \emph{isomorphic} if there exists
a bijective map \(\iota:V_1 \rightarrow V_2\) such that
\(\iota(x) = x\) holds for all \(x \in X\) and 
\(\{u,v\} \in E_1\) if and only if 
\(\{\iota(u),\iota(v)\} \in E_2\) for any two distinct
\(u,v \in V_1\). In case we also have edge-weightings
\(\omega_i:E_i \rightarrow \mathbb{R}\),
\(i \in \{1,2\}\), the weighted phylogenetic trees
\((T_1,\omega_1)\) and \((T_2,\omega_2)\) are
isomorphic if, in addition, 
\(\omega_1(\{u,v\}) = \omega_2(\{\iota(u),\iota(v)\})\)
holds for every edge \(\{u,v\} \in E_1\).
Note that interior edges with weight~0 can give
rise to non-isomorphic weighted phylogenetic trees
that induce the same \(k\)-dissimilarity. Therefore,
in the following we will always implicitly assume
that in any weighted phylogenetic tree interior
edges are assigned positive weights. We call such edge-weightings
\emph{interior-positive}, for short.

We also apply the above terminology 
to (weighted) rooted phylogenetic trees
with the following minor adaptations.
For two rooted
phylogenetic trees \(T_1\) and \(T_2\) with
roots \(\rho_1\) and \(\rho_2\) to be \emph{isomorphic}
we require, in addition, that \(\iota(\rho_1) = \rho_2\) holds.
Note that in a weighted, rooted phylogenetic
tree \((T,\rho,\omega)\) with \(\omega\) equidistant,
every interior edge has a non-negative weight
while pendant edges might have negative weights
(cf. Figure~\ref{figure:rooted:example:intro}(c)).
Again, to avoid non-isomorphic weighted, rooted phylogenetic 
trees giving rise to the same \(k\)-dissimilarity,
we always assume that the edge-weightings are
interior-positive.  A rooted phylogenetic tree \(T=(V,E)\) on \(X\)
with root \(\rho\) is \emph{binary}
if every vertex in \(V \setminus (X \cup \{\rho\})\)
has degree precisely three and \(\rho\) has degree two. 

Next note that for every weighted, rooted 
phylogenetic tree \((T,\rho,\omega)\)
with, not necessarily non-negative, 
equidistant edge-weighting \(\omega\) there exists a 
constant \(M \geq 0\) such that the 
edge-weighting \(\omega_M\), that assigns weight \(\omega(e)\) to every interior edge \(e\) and
weight \(\omega(e) + M\) to every pendant edge \(e\),
is also equidistant and non-negative. Thus,
given a weighted, rooted phylogenetic tree \((T,\rho,\omega)\) on \(X\)
with \(\omega\) equidistant, we can construct, for
any sufficiently large constant \(M \geq 0\),
a weighted phylogenetic tree \((T(M),\omega(M))\) on \(X\)
with \(\omega(M)\) non-negative and interior-positive as follows:
If \(\rho\) has degree at least three, then put
\(T(M) = T\) and \(\omega(M) = \omega_M\).
Otherwise, delete \(\rho\) and
connect the two vertices \(u\) and \(v\) adjacent 
to \(\rho\) by a new edge with weight
\(\omega_M(\{\rho,u\}) + \omega_M(\{\rho,v\})\)
(cf. Figure~\ref{figure:example:unrooting}).
Note that if \(M\) is known, we can completely 
recover \((T,\rho,\omega)\) from \((T(M),\omega(M))\).

\begin{figure}
\centering
\includegraphics[scale=1.0]{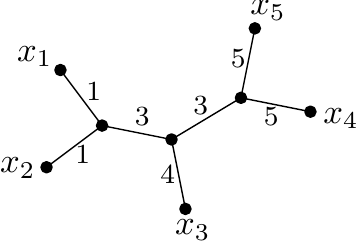}
\caption{The weighted phylogenetic tree \((T(M),\omega(M))\)
         arising from the weighted, rooted phylogenetic
         tree in Figure~\ref{figure:rooted:example:intro}(c) 
         for \(M=2\).}
\label{figure:example:unrooting}
\end{figure}

For every rooted phylogenetic tree
\(T=(V,E)\) on \(X\) with root \(\rho\)
there is a natural partial ordering $\leq_T$ on $V$ 
with unique minimal element $\rho$
defined by $v\leq_T w$ if and only if $v$ is a 
vertex of the unique path from $w$ to $\rho$ in $T$.
The rooted subtree \(T_v\) of \(T\) induced by \(v \in V\)
has vertex set \(\{u \in V : v \leq_T u\}\) and root \(v\).
In addition, for any equidistant edge-weighting 
\(\omega\) of \(T\), we define the \emph{height} \(h_{(T,\omega)}(v)\)
of \(v\), also referred to as the height of \(T_v\), 
as the value \(d_{(T,\omega)}(v,x)\) for any
leaf \(x\) of \(T_v\). Note that this height is well-defined
in view of the fact that \(\omega\) is equidistant. 

Finally, for rooted, as well as unrooted, phylogenetic trees
\(T\) on \(X\) we denote, for any subset \(Y \subseteq X\),
the smallest subtree of \(T\) containing the vertices in \(Y\)
by \(T|_Y\) and refer to it as the \emph{restriction}
of \(T\) to \(Y\). (To formally view $T|_Y$ as a phylogenetic tree
on $Y$, we suppress any vertices of degree~$2$.)
In case \(T\) is rooted, we also consider
\(T|_Y\) as a rooted phylogenetic tree where we
distinguish the minimal element in the vertex set of 
\(T|_Y\) with respect to the partial order \(\leq_T\)
as the root of the restriction. And, in case \(T\)
has an edge-weighting \(\omega\),
we consider \(T|_Y\) as a weighted tree 
with edge-weighting \(\omega|_Y\) 
obtained by restricting \(\omega\) to the edge 
set of \(T|_Y\).

\section{Determining trees} \label{sec:treelike}

We begin this section by stating a uniqueness theorem that
will be useful later:

\begin{theorem}\label{thm:k-diss-root-unique}
For every integer $k \geq 2$ and every set $X$ with at least $2k-1$
elements we have: 
\begin{itemize}
\item[(i)] Two weighted phylogenetic trees 
$(T_1,\omega_1)$ and $(T_2,\omega_2)$ on \(X\) 
with \(\omega_i\) non-negative and interior-positive, \(i \in \{1,2\}\),
are isomorphic if and only if \(D^k_{(T_1,\omega_1)} = D^k_{(T_2,\omega_2)}\) holds.
\item[(ii)] Two weighted, rooted phylogenetic trees 
$(T_1,\rho_1,\omega_1)$ and $(T_2,\rho_2,\omega_2)$ on \(X\) 
with \(\omega_i\) equidistant and interior-positive, \(i \in \{1,2\}\), are
isomorphic if and only if \(D^k_{(T_1,\omega_1)} = D^k_{(T_2,\omega_2)}\) holds.
\end{itemize}
\end{theorem}

Note that for $k=2$ parts (i) and (ii) of this theorem  
are well-known  (see e.g. \citet[Theorem~7.1.8]{sem-ste-03a}). 
Moreover, part (i) is just a restatement of
Theorem~\ref{thm:k-diss-unique} above due to Pachter and Speyer,
and part (ii) immediately follows from part~(i)
by considering the weighted phylogenetic trees 
\((T_1(M),\omega_1(M))\) and \((T_2(M),\omega_2(M))\)
for some sufficiently large constant \(M \geq 0\).

We now prove the first part of Theorem~\ref{main}:

\begin{theorem}\label{thm:2k-treelike}
Let $k\geq 2$ and $D$ be a \(k\)"=dissimilarity map 
on a set \(X\), $|X|\ge 2k$.  
\begin{enumerate}
\item[(i)] \label{thm:2k-treelike:tl}
\(D\) is treelike if and only if
the restriction of \(D\) to every \(2k\)-element
subset of \(X\) is treelike.
\item[(ii)] 
\label{thm:2k-treelike:eq} $D$ is equidistant if and only if
the restriction of \(D\) to every \(2k\)-element
subset of \(X\) is equidistant.
\end{enumerate}
\end{theorem}

\begin{proof}
(i) For \(k=2\) this well-known (see e.g. \cite{sem-ste-03a}). So we
shall assume in the following
that \(k \geq 3\) holds. Clearly, if \(D\) 
is treelike, then also the
restriction to every \(2k\)-element subset 
of \(X\) is treelike.

Conversely, assume that the restriction of \(D\) to
every \(2k\)-element subset of \(X\) is treelike. 
Note that this implies that the restriction
of \(D\) to every \(i\)-element subset $Y$ of 
\(X\), \(k \leq i \leq 2k\),
is treelike, that is, there exists
a weighted phylogenetic tree \((T_Y,\omega_Y)\)
on \(Y\) with \(\omega_Y\) non-negative and
interior-positive such that \(D|_{Y} = D^k_{(T_Y,\omega_Y)}\) holds.

Now consider an arbitrary pair of 
elements \(\{a,b\} \in \binom{X}{2}\).
We claim that in any weighted phylogenetic tree \((T_Z,\omega_Z)\),
\(Z \in \binom{X}{2k-1}\), \(\{a,b\} \subseteq Z\), 
the induced distance \(d_{(T_Z,\omega_Z)}(a,b)\) is the same.
To show this, it suffices to consider such sets 
\(Z, Z'\) with \(Z' = (Z \setminus\{x\}) \cup \{y\}\) for two
distinct elements \(x,y \in X \setminus \{a,b\}\). We now consider
the weighted phylogenetic tree \((T_Y,\omega_Y)\) for the \(2k\)-element
set \(Y:=Z \cup \{y\}\). Since $\card Z=\card Z' = 2k-1$,  
it follows by Theorem~\ref{thm:k-diss-root-unique}(i) that 
$(T_Z,\omega_Z)$ is isomorphic to $(T_Y|_Z,\omega_Y|_Z)$ 
and $(T_{Z'},\omega_{Z'})$ is isomorphic to 
$(T_Y|_{Z'},\omega_Y|_{Z'})$. This implies that the induced
distance between \(a\) and \(b\) is the same for
\((T_Z,\omega_Z)\)  and \((T_{Z'},\omega_{Z'})\), as claimed.

As a consequence, for every pair $\{a,b\}\in\binom{X}{2}$, 
the restriction of \(D\) to any
\((2k-1)\)"=element subset of \(X\) containing $a$ and $b$ 
yields the same distance
between \(a\) and \(b\), which we denote by 
\(\delta(a,b)\). 
Note that the restriction of the so-defined 
2"=dissimilarity \(\delta\) on $X$
to every 4-element subset of \(X\) is treelike:
For any four distinct elements
\(a,b,c,d \in X\) we can select an arbitrary
\(Z \in \binom{X}{2k-1}\) with \(\{a,b,c,d\} \subseteq Z\) and
in the weighted phylogenetic tree \((T_Z,\omega_Z)\) the induced distances
between \(a,b,c,d\) will equal the corresponding values  
of \(\delta\).
Hence, (since the 
theorem holds for $k=2$) there exists a unique weighted 
phylogenetic tree \((T,\omega)\) on \(X\)
with \(\omega\) non-negative and interior-positive such that
\(D^2_{(T,\omega)} = \delta\) holds. Moreover, the restriction of
\((T,\omega)\) to any $2k$-element subset 
\(Z \subseteq X\) is isomorphic to \((T_Z,\omega_Z)\). 
Hence, the \(k\)"=dissimilarity \(D^k_{(T,\omega)}\) 
must be \(D\). \\

\noindent (ii) Again, if a $k$"=dissimilarity on $X$ is 
equidistant, so is 
its restriction to each $2k$"=subset of~$X$. So, let $D$ be a 
$k$"=dissimilarity 
on $X$ such that its restriction to each $Y\in \binom X{2k}$ 
is represented by a weighted, rooted phylogenetic tree $(T_Y,\rho_Y,\omega_Y)$
on \(Y\) with \(\omega_Y\) equidistant and interior-positive. 
Then, for some sufficiently large $M \geq 0$, all the 
weighted phylogenetic trees $(T_Y(M),\omega_Y(M))$ 
are such that \(\omega_Y(M)\) is non-negative and
interior-positive. Therefore,
by the first part of the theorem, 
there exists a unique weighted phylogenetic tree 
$(T,\omega)$ on \(X\) with \(\omega\) non-negative and 
interior-positive such that $D^k_{(T,\omega)}(A)=D(A)+kM$ holds
for all \(A \in \binom{X}{k}\).
Moreover, $(T,\omega)$ must be isomorphic
to \((T'(M),\omega'(M))\) for some
weighted, rooted phylogenetic tree \((T',\rho',\omega')\) on \(X\)
with \(\omega'\) equidistant and interior-positive,
since otherwise 
there would exist some $Y \in \binom{X}{2k}$ such that 
$\omega_Y$ is not equidistant in view of the
fact that, for all $Y \in \binom{X}{2k}$, 
$(T|_Y,\omega|_Y)$ is isomorphic to $(T_Y(M),\omega_Y(M))$. 
Hence $D$ must equal \(D^k_{(T',\omega')}\), as required.
\end{proof}

Using this theorem we now show that, for
fixed $k\ge 3$, it is possible to efficiently check 
when a $k$-dissimilarity is treelike/equidistant.
Note that any algorithm to check
whether a given $k$-dissimilarity $D$ is treelike/equidistant
needs to read $D$ first. Assuming that $D$
is given as the list of values it takes on for
each $k$-element subset of $X$, this yields a
lower bound of \(|X|^k\) on the run-time of any 
such algorithm.

\begin{cor}\label{polycheck}
For any fixed $k\geq 3$ and any 
\(k\)"=dissimilarity \(D\) on \(X\),
there is an algorithm with run-time in $O(f(k) \cdot |X|^{2k})$
to decide whether \(D\) is treelike/equidistant or not, 
where $f$ is a function that does not depend on $|X|$.
\end{cor}

\begin{proof}
Given a \(k\)"=dissimilarity \(D\) on \(X\),
it suffices to check for every \(Z \in \binom{X}{2k}\) whether
\(D|_{Z}\) is treelike/equidistant. To do this, one
can enumerate all (isomorphism classes of) 
unweighted phylogenetic trees (rooted or unrooted)
with \(2k\) leaves labeled
by the elements in~\(Z\). Note that the number of these trees
depends on \(k\) but not on \(|X|\).
For each of those trees \(T\), it remains to check
if there exists an edge-weighting \(\omega\) with certain
properties so that \(D^k_{(T,\omega)} = D|_{Z}\) holds.
The latter can be phrased as a test whether a system
of linear equations and inequalities has a solution,
a problem for which a polynomial time algorithm is known
(see e.g. \cite{sch-98a}). Therefore, one can check
in \(O(f(k))\) time whether \(D|_{Z}\) is treelike/equidistant
where \(f\) is a function that does not depend on $|X|$.
Since the number of $2k$-element 
subsets of $X$ is in $\mathcal{O}(|X|^{2k})$, 
this establishes the required run-time bound.
\end{proof}

\section{Sharpness of the bounds}\label{sec:examples}

In this section, we shall prove the second part of Theorem~\ref{main},
that is, we shall prove that the bounds
presented in the theorem are indeed sharp.
More specifically, for each \(k \geq 3\), we
will present an example of a $k$"=dissimilarity $D$ 
whose restriction to every $(2k-1)$-element subset is
treelike/equidistant while $D$ is not treelike/equidistant. 
These examples will be presented in Examples
\ref{exmp:equidistant} and \ref{exmp:treelike} below.

We begin by presenting a useful lemma.
Assume we have $k \geq 3$, $\card X \geq k$  and that
$(T=(V,E),\rho,\omega)$ is a weighted, rooted phylogenetic 
tree on $X$ with \(\omega\) equidistant and interior-positive.
In addition, assume that $\rho$ is adjacent to precisely two vertices
\(u\) and \(v\) (cf. Figure~\ref{figure:counter:example}(a)). 
Put $a=\omega(\{\rho,u\})$ and $b=\omega(\{\rho,v\})$. 
Now define, for each $\alpha \in \RR$ with 
$\alpha < 2\min\{a,b\}$, a new equidistant
edge-weighting \(\omega_{(k,\alpha)}\) for \(T\) 
(cf. Figure~\ref{figure:counter:example}(b))
by putting, for all $e\in E$,
\[
\omega_{(k,\alpha)}(e)=\begin{cases} \omega(e)+\alpha/k,&\text{if $e$ 
is incident to a leaf},\\
\omega(e)-\alpha/2,&\text{if $e$ is incident to $\rho$},\\ 
\omega(e),&\text{else}\,.\end{cases}
\]

\begin{figure}
\centering
\includegraphics[scale=1.0]{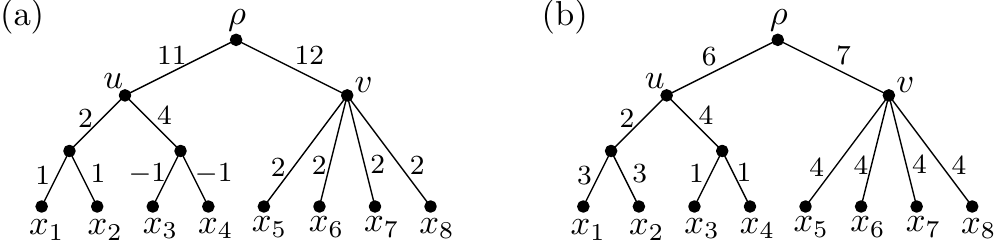}
\caption{(a) A weighted, rooted phylogenetic tree $(T,\rho,\omega)$
             with an equidistant and interior-positive edge-weighting
             \(\omega\).
         (b) The same rooted phylogenetic tree but with the
             equidistant edge-weighting \(\omega_{(k,\alpha)}\)
             for \(k=5\) and \(\alpha=10\). Note that 
             \(D^5_{(T,\omega)} = D^5_{(T,\omega_{(5,10)})}\).}
\label{figure:counter:example}
\end{figure}

\begin{lemma}\label{lem:Tkalpha}
Suppose $k\geq 3$, $\card X \geq k$  and that
$(T=(V,E),\rho,\omega)$ is a weighted, rooted phylogenetic tree 
on $X$ with \(\omega\) equidistant and interior-positive
such that $\rho$ is adjacent with precisely two vertices
\(u\) and \(v\) and put $a=\omega(\{\rho,u\})$ and 
$b=\omega(\{\rho,v\})$.  
Then for each $\alpha\in \RR$ with $\alpha < 2\min\{a,b\}$ 
and $A\in\binom Xk$ we have 
\[ 
D^k_{(T,\omega_{(k,\alpha)})}(A)=
\begin{cases} 
D^k_{(T,\omega)}(A)+\alpha,& \text{if $A \subseteq L(T_u)$ or $A \subseteq L(T_v)$},\\
D^k_{(T,\omega)}(A),&\text{else}\,.
\end{cases}
\]
In particular, $(T,\rho,\omega)$ and 
$(T,\rho,\omega_{(k,\alpha)})$ induce 
the same $k$"=dissimilarity if and only if we have $\alpha=0$ 
or $|L(T_u)| \leq k-1$ and $|L(T_v)| \leq k-1$ hold.
\end{lemma}
\begin{proof}
Let $A\subseteq \binom Xk$. The restriction $T|_A$  
contains $k$ pendant edges. 
Moreover, $T|_A$ contains the two edges incident in \(T\)
with \(\rho\) if and only if we have $A \cap L(T_u) \neq \emptyset$ 
and \(A \cap L(T_v) \neq \emptyset\). 
Hence we have
\(D^k_{(T,\omega_{(k,\alpha)})}(A) = D^k_{(T,\omega)}(A) + k\frac{\alpha}{k}\)
in case $A \subseteq L(T_u)$ or $A \subseteq L(T_v)$ holds and 
we have \(D^k_{(T,\omega_{(k,\alpha)})}(A) = D^k_{(T,\omega)}(A) + k\frac{\alpha}{k}
- 2 \frac{\alpha}{2} = D^k_{(T,\omega)}(A)\) otherwise,
as claimed. The second assertion trivially holds if 
$\alpha=0$. If $\alpha\not=0$, then it holds since all 
$A\in\binom Xk$ contain leaves from both $L(T_u)$ and \(L(T_v)\) 
if and only if $|L(T_u)| \leq k-1$ and $|L(T_v)| \leq k-1$ hold.
\end{proof}

\begin{exmp}[Equidistant]\label{exmp:equidistant}
Let $k\geq3$ and $(T,\rho,\omega)$ be a weighted,
rooted phylogenetic tree on \(X\) with \(\omega\)
equidistant and interior-positive such that 
the root $\rho$ of $T$ is adjacent 
with precisely two vertices \(u\) and \(v\).
Put $a=\omega(\{\rho,u\})$ and 
$b=\omega(\{\rho,v\})$.
Assume that 
$\card{L(T_u)},\card{L(T_v)}\geq k$. 
Now, choose some non"=zero $\alpha < 2\min\{a,b\}$ 
and define a $k$"=dissimilarity $D$ on $X$ via
\[
D(A):=
\begin{cases}
D^k_{(T,\omega_{(k,\alpha)})}(A),&\text{if } A\subseteq L(T_u),\\
D^k_{(T,\omega)}(A),&\text{else,}
\end{cases}\,
\]
for all $A\in \binom Xk$. 
We first show that the restriction of $D$ to 
any $(2k-1)$-element subset of $X$ is equidistant: For any 
$Y\in\binom X {2k-1}$ we define a weighted,
rooted phylogenetic tree $(T_Y,\rho_Y,\omega_Y)$ on $Y$ 
with \(\omega_Y\) equidistant and interior-positive
by setting
\[
(T_Y,\rho_Y,\omega_Y):=
\begin{cases}
(T,\rho,\omega_{(k,\alpha)})|_Y ,&\text{if }\card{Y \cap L(T_u)}\geq k,\\
(T,\rho,\omega)|_Y,&\text{else}.
\end{cases}
\]
By the definition of $D$ it follows that
$D|_Y=D^k_{(T_Y,\omega_Y)}$ holds.

We now show that $D$ is not equidistant. 
It suffices to show that there exists some subset 
$Z$ of $X$ such that $D|_Z$ is not equidistant. So, 
let $Z\subseteq X$ be such that 
$\card{Z\cap L(T_u)}=\card{Z\cap L(T_v)}=k$ (such a 
subset exists since $\card{L(T_u)},\card{L(T_v)}\geq k$), 
and suppose there exists some weighted, rooted
phylogenetic tree \((T_Z,\rho_Z,\omega_Z)\) on \(Z\)
with \(\omega_Z\) equidistant and interior-positive 
such that $D|_Z=D^k_{(T_Z,\omega_Z)}$ holds. For every $a\in Z\cap L(T_u)$ 
we define the weighted, rooted phylogenetic tree
\((T_a,\rho_a,\omega_a):=(T_Z,\rho_Z,\omega_Z)|_{Z\setminus\{a\}}\)
on $Z\setminus\{a\}$. 
Choose distinct $x,y\in Z\cap L(T_u)$
which are not adjacent to a common vertex
of degree 3 (this is possible since $\card{Z \cap L(T_u)} \geq k>2$). 
By Theorem~\ref{thm:k-diss-root-unique}(ii), 
$(T_x,\rho_x,\omega_x)$ is 
isomorphic to 
$(T_{Z\setminus\{x\}},\rho_{Z\setminus\{x\}},\omega_{Z\setminus\{x\}})$ and 
$(T,\rho,\omega)|_{Z\setminus\{x\}}$, and $(T_y,\rho_y,\omega_y)$ 
is isomorphic to 
$(T_{Z\setminus\{y\}},\rho_{Z\setminus\{y\}},\omega_{Z\setminus\{y\}})$ 
and $(T,\rho,\omega)|_{Z\setminus\{y\}}$. 

By our choice of $x$ and $y$, up 
to isomorphism,  there exists only one possible 
weighted, rooted phylogenetic tree on \(Z\) 
whose restriction to $Z\setminus\{x\}$ and 
$Z\setminus\{y\}$ is $T_x$ and $T_y$, 
respectively. Hence $(T_Z,\rho_Z,\omega_Z)$ 
is isomorphic to $(T,\rho,\omega)|_Z$. 
However, this contradicts the fact that $(T_Z,\rho_Z,\omega_Z)$ 
induces $D|_Z$ since 
$D(Z\cap L(T_u))=D^k_{(T,\omega_{(k,\alpha)})}(Z \cap L(T_u))
\not=D^k_{(T,\omega)}(Z\cap L(T_u))=D^k_{(T_Z,\omega_Z)}(Z\cap L(T_u))$, 
where the first equality is by 
definition and the second follows from the 
fact that $(T_Z,\rho_Z,\omega_Z)$ 
and $(T,\rho,\omega)|_Z$ are isomorphic.
\end{exmp}

\begin{exmp}[Treelike]\label{exmp:treelike}
An example for the case $k=3$ was given by \cite{che-fic-07}. Here we
give an example for general $k$.
Based on the weighted, rooted phylogenetic tree \((T,\rho,\omega)\)
in Example~\ref{exmp:equidistant}, consider $(T(M),\omega(M))$ 
for some sufficiently large constant \(M \geq 0\) and choose \(\alpha > 0\).
Then, using 
the same arguments as in Example~\ref{exmp:equidistant}, 
it is straight-forward to show
that the $k$"=dissimilarity $D$ constructed in the same
way as in Example~\ref{exmp:equidistant}
is not treelike while its restriction to every 
$(2k-1)$-element subset of $X$ is treelike.
\end{exmp}

\section{The case $2k-1$ for equidistant $k$-dissimilarities}
\label{equisection}

It is well-known that if the restriction of a 2-dissimilarity $D$ on $X$
to every subset of $X$ of size $3$ is equidistant, then $D$ is equidistant 
\cite[Theorem 7.2.5]{sem-ste-03a}.
In contrast, in Example~\ref{exmp:equidistant}, we have seen
that for $k\ge 3$ a $k$"=dissimilarity $D$ is not necessarily 
equidistant if its restriction to every 
$(2k-1)$-element subset is equidistant. However,
we shall now prove that we can still recover a tree if 
the restriction of such a $D$ to all $(2k-1)$-element 
subsets \(Y \subseteq X\) 
is induced by a weighted, rooted phylogenetic tree
\((T,\rho,\omega)\) on \(Y\) with \(\omega\) equidistant and 
interior-positive such that \((T,\rho,\omega)\) is \emph{generic},
that is, \(T\) is binary and no two distinct interior vertices 
have the same height.

\begin{theorem}\label{thm:topology}
Let $k\geq 3$ and $D$ be a $k$"=dissimilarity map on 
$X$ such that, for all $Y\in\binom X{2k-1}$, there
exists a generic weighted, rooted phylogenetic tree
\((T_Y,\rho_Y,\omega_Y)\) on \(Y\) with 
\(\omega_Y\) equidistant and interior-positive
such that \(D|_Y = D_{(T_Y,\omega_Y)}\) holds.
Then there exists a binary rooted phylogenetic
tree $T$ on $X$ such that, for all $Y\in\binom X{2k-1}$,
the unweighted, rooted phylogenetic trees $T_Y$ and $T|_Y$ are isomorphic.
\end{theorem}

To prove this theorem, we shall use a well-known result 
about collections of rooted phylogenetic trees 
each having three leaves,
that will allow us to ``merge'' trees, which 
we now recall. A \emph{triplet} on $X$ is a 
pair $(\{a,b\},c)$ 
with $a,b,c\in X$ distinct, which we denote also by $ab|c$.
The set of all triplets on $X$ is denoted by $\cR(X)$, 
and a subset $\cR$ of $\cR(X)$ is called a \emph{triplet system} 
on $X$. Given a rooted phylogenetic tree $T$ on $X$, 
the triplet system $\cR_T$ of $T$ is the set of all triplets $ab|c$ 
on $X$ such that the path from $a$ to $b$ in $T$ is 
vertex-disjoint from the path from $c$ to the root $\rho$ 
in $T$. It is easily seen that for a rooted phylogenetic tree $T$ on $X$ 
and a rooted phylogenetic tree $T'$ on $Y\subseteq X$, we have 
$\cR_{T'}\subseteq \cR_T$ if  $T'$ 
is isomorphic to $T|_Y$. We now state the 
aforementioned result:

\begin{theorem}[Theorem~9.2~(ii) in 
\cite*{dre-hub-11a}]\label{thm:triplets}
A rooted phylogenetic tree \(T\) on \(X\) is, up to isomorphism,
uniquely determined by the triplet system $\cR_T$.
Moreover, given a triplet system $\cR\subseteq\cR(X)$ 
there exists a rooted phylogenetic tree on $X$ with $\cR_T=\cR$ 
if and only if $\cR$ satisfies the following two conditions:
\begin{enumerate}[($\cR$1)]
\item For any three elements $a,b,c\in X$ at 
most one of the triplets $ab|c$, $bc|a$ 
and $ca|b$ is contained in $\cR$.
\item For any four elements $a,b,c,d\in X$, $ab|c\in \cR$ 
implies $ad|c\in\cR$ or $ab|d\in\cR$.
\end{enumerate}
\end{theorem}

To prove Theorem~\ref{thm:topology}
we will also use the 
following rather technical result:

\begin{lemma}
\label{prop:2k-2-unique}
Let \(k \geq 3\) be an integer, \(X\) a set with \(|X| = 2k-2\)
and \((T_1,\rho_1,\omega_1)\) and \((T_2,\rho_2,\omega_2)\) be two generic 
weighted, rooted phylogenetic trees on \(X\) 
with \(\omega_i\) equidistant and interior-positive, \(i \in \{1,2\}\).
If \(D^{k}_{(T_1,\omega_1)} = D^{k}_{(T_2,\omega_2)}\) then \(T_1\) and \(T_2\)
are isomorphic as unweighted, rooted phylogenetic trees.
\end{lemma}

\begin{proof}
We distinguish two cases.
First consider the case that at least one of \(T_1\) and \(T_2\), say \(T_1\),
contains a vertex \(v_1\) such that the set $B$ of leaves of the rooted subtree \((T_1)_{v_1}\) 
has cardinality \(k-1\). Define \(A := X \setminus B\).
We claim that \(T_2\) must contain a vertex \(v_2\)
such that the set of leaves of \((T_2)_{v_2}\) is \(B\).
To establish this, first note that 
\(D^k_{(T_1,\omega_1)}(A \cup \{b\}) = D^k_{(T_1,\omega_1)}(A \cup \{b'\})\)
and, therefore, in view of \(D^{k}_{(T_1,\omega_1)} = D^{k}_{(T_2,\omega_2)}\),
also \(D^k_{(T_2,\omega_2)}(A \cup \{b\}) = D^k_{(T_2,\omega_2)}(A \cup \{b'\})\)
must hold for all \(b,b' \in B\). This implies
that, for every \(b \in B\), the height of the
vertex \(w_b\) where the path from \(b\) to the
root of \(T_2\) first meets the subtree \(T_2|_A\)
must be the same. Hence, since \(T_2\) is generic, 
\(w_b = w_{b'}\) holds for all \(b,b' \in B\).
But then the tree \(T_2|_{B}\) must equal the
tree \((T_2)_{v_2}\) for some vertex \(v_2\) of \(T_2\),
as claimed.

Next, we claim that \(T_1|_A\) and \(T_2|_A\)
as well as \(T_1|_B\) and \(T_2|_B\) are isomorphic
as unweighted, rooted phylogenetic trees. By Theorem~\ref{thm:triplets}
it suffices to show that \(\cR(T_1|_A) = \cR(T_2|_A)\)
and \(\cR(T_1|_B) = \cR(T_2|_B)\) holds. 
In the following we will focus on the set~\(A\). 
A completely analogous argument  
yields \(\cR(T_1|_B) = \cR(T_2|_B)\).
So, consider three arbitrary distinct elements
\(a\), \(b\) and \(c\) in \(A\) and an arbitrary
\((k-2)\)-element subset \(C\) of \(B\). Up to
relabeling, Figure~\ref{figure:proofpatch}
depicts the possible cases for the structure
of the tree \(T_1|_{C \cup \{a,b,c\}}\). 
Note that, by the assumption that \(T_1\) is generic,
cases (b), (d), (g), (i), (j) and (k) are ruled out.
In the remaining cases we have:
\begin{itemize}
\item[(a)]
\(D^k_{(T_1,\omega_1)}(C \cup \{a,b\}) 
= D^k_{(T_1,\omega_1)}(C \cup \{a,c\})
< D^k_{(T_1,\omega_1)}(C \cup \{b,c\})\)
\item[(c)] (and, similarly, (e) and (f))\\
\(D^k_{(T_1,\omega_1)}(C \cup \{b,c\})
< D^k_{(T_1,\omega_1)}(C \cup \{a,b\}) 
= D^k_{(T_1,\omega_1)}(C \cup \{a,c\})\)
\item[(h)]
\(D^k_{(T_1,\omega_1)}(C \cup \{b,c\})
< D^k_{(T_1,\omega_1)}(C \cup \{a,c\})
< D^k_{(T_1,\omega_1)}(C \cup \{a,b\})\)
\end{itemize}
So, we have \(ab|c \in \mathcal{R}(T_1|_A)\)
if and only if either
\begin{itemize}
\item[(1)]
\(D^k_{(T_1,\omega_1)}(C \cup \{a,c\}) 
= D^k_{(T_1,\omega_1)}(C \cup \{b,c\})
\neq D^k_{(T_1,\omega_1)}(C \cup \{a,b\})\) holds, or
\item[(2)]
\(D^k_{(T_1,\omega_1)}(C \cup \{a,b\})\),
\(D^k_{(T_1,\omega_1)}(C \cup \{a,c\})\) and
\(D^k_{(T_1,\omega_1)}(C \cup \{b,c\})\) are pairwise
distinct and \(D^k_{(T_1,\omega_1)}(C \cup \{a,b\})\)
is the smallest value among them.
\end{itemize}
But this implies, in view of the fact that
\(D^k_{(T_1,\omega_1)} = D^k_{(T_2,\omega_2)}\)
holds, that we have 
\(\mathcal{R}(T_1|_A) = \mathcal{R}(T_2|_A)\), as required.

\begin{figure}
\centering
\includegraphics[scale=1.0]{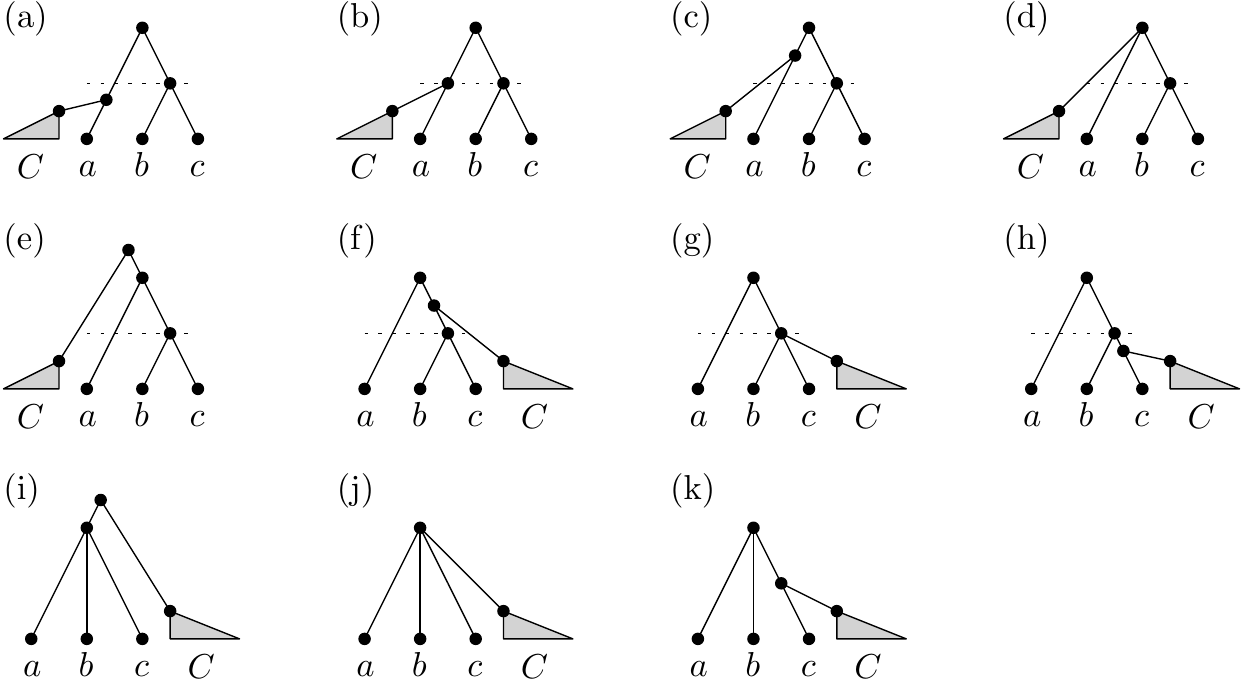}
\caption{Schematic representations of the cases considered
         in the proof of Lemma~\ref{prop:2k-2-unique} in the context of
         establishing that \(\cR(T_1|_A) = \cR(T_2|_A)\) holds.}
\label{figure:proofpatch}
\end{figure}

Now, to show that the trees \(T_1\) and \(T_2\)
are isomorphic as unweighted, rooted phylogenetic trees, 
it remains to show that \(T_1|_B\) and \(T_2|_B\) are attached to
\(T_1|_A\) and \(T_2|_A\), respectively, at the same
position. To establish this, consider the set
\(A^*\) containing all \(a \in A\)
such that \(D^k_{(T_1,\omega_1)}(B \cup \{a\})\) is minimal.
Note that there must exist vertices \(w_1\) in \(T_1\) and
\(w_2\) in \(T_2\) such that \(T_1|_{A^*} = (T_1)_{w_1}\) and
\(T_2|_{A^*} = (T_2)_{w_2}\) hold. Moreover, note that \(v_1\) and \(w_1\)
as well as \(v_2\) and \(w_2\) must be adjacent to
a common vertex, namely the vertex where \(T_1|_B\) and \(T_2|_B\) 
are attached to
\(T_1|_A\) and \(T_2|_A\), respectively. But this
implies that \(T_1\) and \(T_2\)
are isomorphic as unweighted, rooted phylogenetic trees.

Next consider the case that neither \(T_1\) nor \(T_2\)
contains a vertex such that the subtree induced by that
vertex has precisely \(k-1\) leaves. Choose some
\(M \in \mathbb{R}\) large enough 
and consider the weighted phylogenetic trees 
\((T_1(M),\omega_1(M))\) and \((T_2(M),\omega_2(M))\).
Note that, for every edge \(e\) of \(T_j(M)\), \(j \in \{1,2\}\),  
removing \(e\) from \(T_j(M)\) yields two subtrees, 
one of which has at least \(k\) leaves. 
But this is the crucial property used in the proof
of Theorem~\ref{thm:k-diss-unique} 
presented 
in \cite{pac-spe-04a}, and the lower bound \(|X| \geq 2k-1\) 
stated in this theorem is only needed to ensure 
that this property holds.
Hence, even in case \(|X| = 2k-2\) the proof can be applied
as long as we have this property. This implies that
\(T_1(M)\) and \(T_2(M)\) are isomorphic
(even as weighted phylogenetic trees!) and, hence, also 
\(T_1\) and \(T_2\), as required.
\end{proof}

\begin{proof}[Proof of {Theorem~\ref{thm:topology}}]
Define $\cR:=\bigcup_{Y\in\binom X{2k-1}} \cR_{T_Y}$. 
We have to show that  Conditions~($\cR$1)
and~($\cR$2) hold for $\cR$.

To show that ($\cR$1) holds, it suffices to show that for 
all $Z\in\binom X{2k-2}$, and distinct $x,y\in X \setminus Z$,
$a,b,c\in Z$ at most one of the triplets $ab|c$, $bc|a$ 
and $ca|b$ is contained in $\cR_{T_{Z\cup\{x\}}}\cup\cR_{T_{Z\cup\{y\}}}$. 
Lemma~\ref{prop:2k-2-unique} implies
that the phylogenetic trees $T^\star:=T_{\{Z\cup\{x\}\}}|_Z$ 
and $T_{Z\cup\{y\}}|_Z$ are isomorphic. By 
Theorem~\ref{thm:triplets}, Condition~($\cR$1) 
holds for $\cR_{T^\star}$, as required.

We now show that ($\cR$2) holds. Let $a,b,c,d$ be
distinct elements of $X$
and suppose $ab|c\in \cR$. Then there exists some 
$Y\in\binom X{2k-1}$ such that $ab|c\in \cR_{T_Y}$. 
If $d\in Y$, then we have $ad|c\in\cR$ or $ab|d\in\cR$ 
since ($\cR$2) holds for $\cR_{T_Y}$. If $d\not\in Y$, 
take some $x\in Y\setminus\{a,b,c\}$ and define 
$Y':=(Y\setminus\{x\})\cup\{d\}$. Again, 
by Lemma~\ref{prop:2k-2-unique}, 
the phylogenetic trees $T_{Y}|_{Y\setminus\{x\}}$ 
and $T_{Y'}|_{Y\setminus\{x\}}$ are isomorphic, 
hence $ab|c$ is also an element of $\cR_{T_{Y'}}$ 
and hence  $ad|c\in\cR$ or $ab|d\in\cR$ since ($\cR$2) 
holds for $\cR_{T_{Y'}}$.
\end{proof}

\begin{remark}
We suspect, but have not been able to prove, 
that, for $k\ge 3$,  if a $k$"=dissimilarity on 
$X$ is induced by an {\em arbitrary} 
weighted, rooted phylogenetic tree $(T_Y,\rho_Y,\omega_Y)$ 
with \(\omega_Y\) equidistant and interior-positive
for all  $Y\in\binom X{2k-1}$, 
then it determines a rooted phylogenetic tree $T$ on 
$X$ such that, for all $Y\in\binom X{2k-1}$, 
the tree $T_Y$ is isomorphic to $T|_Y$.
Furthermore, depending on the topology of the 
unweighted phylogenetic tree $T$ arising in this way, it might 
even still be possible to 
assign weights to the edges of $T$ so that the edge-weighting
is equidistant and the induced \(k\)-dissimilarity is~$D$. 
For example, if the  
number of leaves in one of the subtrees induced by
the vertices adjacent to the root of $T$ is smaller than $k$,
then one can extend the arguments 
in the proof of Theorem~\ref{thm:2k-treelike} to 
show that one can indeed construct a suitable edge-weighting 
for $T$, and hence $D$ is equidistant in this case.
\end{remark}

\section{3-dissimilarities}\label{sec:k=3}

In this section, we prove that treelike
and equidistant 3-dissimilarities 
can be characterized by certain 6-point conditions.
We begin by recalling some 
conditions for  characterizing treelike and
equidistant 2-dissimilarities (see e.g.  
\cite{smo-62,zar-65,bun-71,gor-87}; and \cite{sem-ste-03a}).

It is well-known that a 
2-dissimilarity  \(D\) on~$X$ is treelike 
if and only if \(D\) is non-negative, it
satisfies the \emph{triangle inequality} (i.e., 
\(D(x_1,x_3) \leq D(x_1,x_2) + D(x_2,x_3)\) holds for any
three distinct elements \(x_1,x_2,x_3 \in X\)) and
\begin{equation}
\label{equation:four:point:condition}
D(x,x')+D(y,y') \le \max\{D(x,y)+D(x',y'),D(x,y')+D(x',y)\}
\end{equation}
holds for any four distinct $x,x',y,y' \in X$. 
Similarly, it is known that $D$ is equidistant if and only if 
\begin{equation}
\label{equation:ultrametric:condition}
D(x,y) \le \max\{D(x,z),D(z,y)\}
\end{equation}
holds for any three distinct $x,y,z \in X$.

Inequalities (\ref{equation:four:point:condition}) and
(\ref{equation:ultrametric:condition}) are commonly called the 
{\em 4-point} and {\em ultrametric} conditions, respectively.
Note that non-negativity of $D$ and the triangle
inequality follow from the 4-point condition
if one defines \(D(a,a) = 0\) for all
\(a \in X\) and then drops the requirement that
the elements are pairwise distinct. However,
as we view a \(k\)-dissimilarity as being a
map from \(\binom{X}{k}\) into \(\mathbb{R}\),
we need to explicitly require
these additional properties.

We now present similar 
conditions that characterize 
treelike/equidistant \(3\)"=dissimilarities that are obtained
by associating with every \(3\)"=dissimilarity a suitable
\(2\)"=dissimilarity. The construction of this \(2\)"=dissimilarity
is similar to the approach followed in the context of 
so-called \emph{perimeter models} considered, for example, in
\cite{hei-ben-97} and \cite{che-fic-07}.

\begin{theorem}\label{thm:6-point-condition}
Let $D$ be a \(3\)"=dissimilarity on 
a set \(X\) with \(|X| \geq 5\).
\begin{enumerate}
\item[(i)] \(D\) is treelike if and only if for all 
$\{a,b,c,d,e\} \in \binom{X}{5}$
\begin{equation}
\label{treelike:non:negative}
\begin{split}
 D(a,c,d)+D(a,c,e)+D(a,d,e)+D(b,c,d)+D(b,c,e)+D(b,d,e)\\
\leq 2\left( D(a,b,c)+D(a,b,d)+D(a,b,e)+D(c,d,e)\right),
\end{split}
\end{equation}
\begin{align}
\label{treelike:from:triangle:inequality}
\begin{split}
2\left(D(a,c,d) + D(a,c,e) + D(b,d,e)\right) 
\leq D(a,b,c) + D(a,b,d) + D(a,b,e)\\
+D(a,d,e) + D(b,c,d) + D(b,c,e) + D(c,d,e), 
\end{split}
\end{align}
\begin{equation}
\label{treelike:from:4:point}
D(a,c,d) + D(a,c,e) + D(b,d,e) \leq \max \left \{
\begin{matrix}
D(a,b,d) + D(a,b,e) + D(c,d,e)\\
D(a,d,e) + D(b,c,d) + D(b,c,e)
\end{matrix}
\right \}\,,
\end{equation}
and for all $\{a,b,c,d,e,e'\} \in \binom{X}{6}$
\begin{align}
\label{forced:equality}
\begin{split}
2D(a,b,e)-D(a,c,e)-D(a,d,e)-D(b,c,e)-D(b,d,e)+2D(c,d,e) = \\
2D(a,b,e')-D(a,c,e')-D(a,d,e')-D(b,c,e')-D(b,d,e')+2D(c,d,e').
\end{split}
\end{align}
\item[(ii)] $D$ is equidistant if and only if 
for all $\{a,b,c,d,e\} \in \binom{X}{5}$
\begin{equation}
\label{ultrametric:five:point}
D(a,b,e) + D(c,d,e) \leq \max \left \{
\begin{matrix}
D(a,c,e) + D(b,d,e) \\
D(a,d,e) + D(b,c,e) 
\end{matrix}
\right \}
\end{equation}
and for all $\{a,b,c,d,e,e'\} \in \binom{X}{6}$ 
Equation~\eqref{forced:equality} holds.
\end{enumerate}
\end{theorem}

\begin{proof}
For any \(Y = \{a,b,c,d,e\} \in \binom{X}{5}\) we define
a map \(\delta_Y:\binom{Y}{2} \rightarrow \mathbb{R}\) as
follows. First define the vector
\[
v_Y := (D(a,b,c),D(a,b,d),D(a,b,e),D(a,c,d),\dots,D(c,d,e))^T
\]
as well as the following matrix and its inverse (note
that $A$ has full rank):
\begin{center}
\scriptsize
\(
A = 
\begin{pmatrix}
1 & 1 & 0 & 0 & 1 & 0 & 0 & 0 & 0 & 0\\
1 & 0 & 1 & 0 & 0 & 1 & 0 & 0 & 0 & 0\\
1 & 0 & 0 & 1 & 0 & 0 & 1 & 0 & 0 & 0\\
0 & 1 & 1 & 0 & 0 & 0 & 0 & 1 & 0 & 0\\
0 & 1 & 0 & 1 & 0 & 0 & 0 & 0 & 1 & 0\\
0 & 0 & 1 & 1 & 0 & 0 & 0 & 0 & 0 & 1\\
0 & 0 & 0 & 0 & 1 & 1 & 0 & 1 & 0 & 0\\
0 & 0 & 0 & 0 & 1 & 0 & 1 & 0 & 1 & 0\\
0 & 0 & 0 & 0 & 0 & 1 & 1 & 0 & 0 & 1\\
0 & 0 & 0 & 0 & 0 & 0 & 0 & 1 & 1 & 1\\
\end{pmatrix}
\)
\quad \quad \quad
\(
A^{-1} = \frac{1}{6} \cdot
\begin{pmatrix}
2 & 2 & 2 & -1 & -1 & -1 & -1 & -1 & -1 & 2\\
2 & -1 & -1 & 2 & 2 & -1 & -1 & -1 & 2 & -1\\
-1 & 2 & -1 & 2 & -1 & 2 & -1 & 2 & -1 & -1\\
-1 & -1 & 2 & -1 & 2 & 2 & 2 & -1 & -1 & -1\\
2 & -1 & -1 & -1 & -1 & 2 & 2 & 2 & -1 & -1\\
-1 & 2 & -1 & -1 & 2 & -1 & 2 & -1 & 2 & -1\\
-1 & -1 & 2 & 2 & -1 & -1 & -1 & 2 & 2 & -1\\
-1 & -1 & 2 & 2 & -1 & -1 & 2 & -1 & -1 & 2\\
-1 & 2 & -1 & -1 & 2 & -1 & -1 & 2 & -1 & 2\\
2 & -1 & -1 & -1 & -1 & 2 & -1 & -1 & 2 & 2\\
\end{pmatrix}
\)
\end{center}
Then, using the notation 
\[
u_Y = (\delta_Y(a,b),\delta_Y(a,c),\delta_Y(a,d),
\delta_Y(a,e),\delta_Y(b,c),\dots,\delta_Y(d,e))^T,
\]
the map \(\delta_Y\) is defined by the unique
solution of the system of linear equations
\begin{equation}
\label{equation:relating:2:and:3}
2 v_Y = A \cdot u_Y.
\end{equation}
In particular we have:
\begin{align}
3 \delta_Y(a,b) &= 2D(a,b,c)+2D(a,b,d)+2D(a,b,e) -D(a,c,d)-D(a,c,e)-D(a,d,e)\notag\\
            &\qquad         -D(b,c,d)-D(b,c,e)-D(b,d,e)+2D(c,d,e), \notag\\
3 \delta_Y(a,c) &= 2D(a,b,c)-D(a,b,d)-D(a,b,e) +2D(a,c,d)+2D(a,c,e)-D(a,d,e)\notag\\
            &\qquad         -D(b,c,d)-D(b,c,e)+2D(b,d,e)-D(c,d,e), \notag\\
3 \delta_Y(a,d) &= -D(a,b,c)+2D(a,b,d)-D(a,b,e) +2D(a,c,d)-D(a,c,e)+2D(a,d,e) \notag\\
            &\qquad         -D(b,c,d)+2D(b,c,e)-D(b,d,e)-D(c,d,e), \label{translate:from:delta}\\
3 \delta_Y(b,c) &= 2D(a,b,c)-D(a,b,d)-D(a,b,e) -D(a,c,d)-D(a,c,e)+2D(a,d,e)\notag\\
            &\qquad         +2D(b,c,d)+2D(b,c,e)-D(b,d,e)-D(c,d,e), \notag\\
3 \delta_Y(b,d) &= -D(a,b,c)+2D(a,b,d)-D(a,b,e) -D(a,c,d)+2D(a,c,e)-D(a,d,e)\notag\\
            &\qquad         +2D(b,c,d)-D(b,c,e)+2D(b,d,e)-D(c,d,e), \notag\\
3 \delta_Y(c,d) &= -D(a,b,c)-D(a,b,d)+2D(a,b,e) +2D(a,c,d)-D(a,c,e)-D(a,d,e)\notag\\
            &\qquad         +2D(b,c,d)-D(b,c,e)-D(b,d,e)+2D(c,d,e).\notag
\end{align}
It is not hard to see that \(D|_{Y}\) 
is treelike/equidistant
if and only if \(\delta_Y\) is treelike/equidistant.
This is the key observation that 
will allow us to translate
the 4-point/ultrametric condition 
characterizing when a 
2"=dissimilarity is treelike/equidistant
into conditions for when a 
\(3\)"=dissimilarity is  treelike/equidistant.

Before we do this, we
need some condition that ensures that, for any two
distinct \(a,b \in X\) and any two distinct 
\(Z,Z' \in \binom{X}{5}\) with \(\{a,b\} \subseteq Z \cap Z'\), we have 
\(\delta_{Z}(a,b) = \delta_{Z'}(a,b)\).
Clearly, it suffices to consider such sets \(Z,Z' \in \binom{X}{5}\) with
\(|Z \cap Z'| = 4\). In particular, for \(Z=\{a,b,c,d,e\}\) 
and \(Z' = (Z \setminus \{e\}) \cup \{e'\}\)
we obtain Equation~(\ref{forced:equality}) which ensures that 
the map \(\delta:\binom{X}{2} \rightarrow \mathbb{R}\)
defined by putting \(\delta(a,b) := \delta_Z(a,b)\) for an
arbitrary \(Z \in \binom{X}{5}\) with \(\{a,b\} \subseteq Z\) 
is well-defined. And in this case 
\(\delta_Z\) is treelike/equidistant
for all 
\(Z \in \binom{X}{5}\) if and only if \(\delta\) 
is treelike/equidistant
if and only if \(D\) is treelike/equidistant. 

We now prove the two assertions of the theorem:
(i) 
Using the 
equations in~(\ref{translate:from:delta}),
it is not hard to check that 
the conditions for \(\delta\) (to be non-negative,
to satisfy the triangle inequality and the
4-point condition) translate into
Inequalities~(\ref{treelike:non:negative}),
(\ref{treelike:from:triangle:inequality}) and 
(\ref{treelike:from:4:point}), respectively.
(ii)
Again it is not hard to check that, using the 
equations in~(\ref{translate:from:delta}), 
the ultrametric condition on \(\delta\) translates into 
Inequality~(\ref{ultrametric:five:point}).
\end{proof}

Note that \citet[Theorem~3.2]{boc-col-09a}   
showed that there exists a family of maps $\phi_k$ from 
the set of all treelike 2-dissimilarities to the set of all 
treelike $k$"=dissimilarities that maps a
2-dissimilarity $D^2_{(T,\omega)}$ induced by a 
weighted phylogenetic tree $(T,\omega)$ on \(X\) with \(\omega\)
non-negative to the treelike $k$"=dissimilarity 
$\phi_k(D^2_{(T,\omega)})=D^k_{(T,\omega)}$. 
For $k=3$ this map can be thought of as multiplication 
with the matrix $A$ as considered in 
the proof of Theorem~\ref{thm:6-point-condition}, 
but for $k\geq 4$ it appears that 
no such simple representation is possible.

Indeed, the key observation used in proving the above
result is that the restriction of any treelike/equidistant
$3$"=dissimilarity $D$ to every 5-element subset $Y \subseteq X$
can be related to a unique treelike/equidistant 2"=dissimilarity
on $Y$ by a system of linear equations
(as this allowed
the straight-forward translation of the 4-point/ultrametric condition
into a 5-point condition).
Unfortunately, it seems that there are
problems when we try to apply this idea in case 
$k \geq 4$, even for $k=4$. 

More specifically, first note that,
although the restriction of any treelike
$4$"=dissimilarity $D$ to every 6-element subset $Y \subseteq X$
can be related to a treelike 2"=dissimilarity on~$Y$ by
a system of linear equations, to do this one has to select a suitable
ordering of the elements in~$Y$ (see also \citet[Theorem~2.2]{boc-col-09a}). 
In contrast, in the case $k=3$ any ordering 
works. Moreover, the system of linear equations does not need
to have a unique solution, that is, even after fixing
a suitable ordering, there can be more than one
2"=dissimilarity on $Y$ associated to the $4$"=dissimilarity
$D$. This further complicates the translation of the 4-point
condition into some form of 
8-point condition for when $D$ is treelike.

\bibliographystyle{jclas}
\bibliography{k-dissimilarity}

\end{document}